\documentclass[12pt,reqno,a4paper]{amsart}

\usepackage{times}

\usepackage[
  margin=1in
]{geometry}

\usepackage{amssymb,amsfonts,amsthm}
\usepackage{marginnote}
\usepackage{comment}
\usepackage{enumitem}
\usepackage{bbm} 

\usepackage{graphicx}

\newcommand{\R}{\mathbb{R}}

\usepackage[dvipsnames]{color}

\usepackage[colorlinks=true, pdfstartview=FitV, linkcolor=black, citecolor=blue, urlcolor=blue]{hyperref}
\DeclareMathOperator*{\esssup}{ess\,sup}
\DeclareMathOperator*{\supp}{supp}
\DeclareMathOperator*{\dist}{dist}

\newtheorem{theorema}{Theorem}

\newtheorem*{Theorem*}{Theorem}
\newtheorem{Lemma}{Lemma}[section]
\newtheorem{Proposition}{Proposition}[section]

\numberwithin{equation}{section}

\definecolor{darkgreen}{rgb}{0,0.5,0}
\definecolor{darkblue}{rgb}{0,0,0.7}
\definecolor{darkred}{rgb}{0.9,0.1,0.1}
\definecolor{lightblue}{rgb}{0,0.51,1}

\title[Blow-up of dynamically restricted critical norms]{Blow-up of dynamically restricted critical norms near a potential Navier-Stokes singularity}

\author[T. Barker]{Tobias Barker}
\address[T. Barker]{Department of Mathematical Sciences, University of Bath, Bath BA2 7AY, UK}
\email{tobiasbarker5@gmail.com}

\author[P. G. Fern\'andez-Dalgo]{Pedro Gabriel Fern\'andez-Dalgo}
\address[P. G. Fern\'andez-Dalgo]{Cergy Paris Universit\'e, Laboratoire de Math\'ematiques AGM, UMR CNRS 8088, France}
\email{pedro.fernandez-dalgo@cyu.fr}

\author[C. Prange]{Christophe Prange}
\address[C. Prange]{Cergy Paris Universit\'e, Laboratoire de Math\'ematiques AGM, UMR CNRS 8088, France}
\email{christophe.prange@cyu.fr}

\date{\today}

\begin{document}

\maketitle

\begin{abstract}
In this paper we develop new methods to obtain regularity criteria for the three-dimensional Navier-Stokes equations in terms of dynamically restricted endpoint critical norms: the critical Lebesgue norm in general or the critical weak Lebesgue norm in the axisymmetric case. This type of results is inspired in particular by a work of Neustupa (2014), which handles certain non endpoint critical norms. Our work enables to have a better understanding of the nonlocal effect of the pressure on the regularity of the solutions.
\end{abstract}

\vspace{1cm}
 

\section{Introduction}
\label{intro}

One currently active research direction on the three-dimensional Navier-Stokes equations
\begin{equation}
\label{ns1intro}
    \partial_t v - \Delta v+ v \cdot \nabla v +\nabla p=0, \phantom{space} \nabla \cdot v = 0.
\end{equation}
is the investigation of scales that play a role in potential singularity formation. Many results were proved in the past ten years that: (i) show dynamically restricted regularity criteria and on the other side of the coin (ii) show that certain scale-invariant quantities accumulate/concentrate on specific scales near a potential singularity. Let us cite (non exhaustively) the following works: \cite{CS14,BG17,ChesDai19,Luo19} on frequency localized results, \cite{LOW18,GrujicXu2019-dynrestr,BP20norm,KangMiuraTsai20-concL2,bradshaw2022local} on spatially localized results. For a more extensive bibliography, we refer to the recent survey \cite{BP22survey}. 

In this vein, let us mention in particular the results of Ne\v{c}as and Neustupa \cite{NN02} and of Neustupa \cite{Neus12,Neus14}. The papers \cite{NN02,Neus12} lead to \cite{Neus14}, where the author proves a dynamically restricted version of the Lady\v{z}enskaja-Prodi-Serrin regularity criteria, namely if $v$ is a finite-energy weak solution to \eqref{ns1intro} in $\R^3\times(-1,0)$ such that $0$ is a first-time singularity and such that $(0,0)$ is a singular point,\footnote{By definition, the point $(0,0)$ is a \emph{singular point} if for all $r\in(0,1)$, $v\notin L^\infty(Q_{(0,0)}(r))$, where $Q_{(0,0)}(r)$ is the parabolic cylinder $B_0(r)\times(-r^2,0)$. The point $(0,0)$ is called a \emph{regular point} if it is not a singular point.} then 
\begin{equation*}
\int\limits_{-1}^0\Bigg(\int\limits_{B_{0}(\sqrt{a})\setminus B_{0}(\sqrt{-at})}|v|^p\, dx\Bigg)^\frac{q}{p}\, ds=\infty
\end{equation*}
for critical Lebesgue exponents $\frac2q+\frac3p=1$, $3\leq q<\infty$, $3<p\leq 9$ and a certain $a>0$ (see Footnote \ref{foot.a}). In this paper, we address the endpoint $p=3$ and $q=\infty$. Hence, the result stated below generalizes in particular the celebrated result of Escauriaza, Seregin and \v{S}ver\'{a}k \cite{ESS2003}. In the case that $v$ is axisymmetric we obtain strengthened results, which generalize the results of \cite{chen2008lower,chen2009lower,koch2009liouville,sersve2009type}. For an illustration of Neustupa's result and our theorem below, we refer to Figure \ref{fig.paraboloid}.

\begin{theorema}\label{theom.main}
Let $v$ be a finite-energy weak solution to the Navier-Stokes equations \eqref{ns1intro} in $\R^3\times(-1,0)$ such that $0$ is a first-time singularity.\footnote{In particular, $v\in C^\infty(-1,T;C^\infty(\R^3))$ for all $T\in(-1,0)$. We introduce this assumption in order to remove certain technicalities. Notice that the framework of first-time singularities is relevant for the study of the global regularity problem for the 3D Navier-Stokes equations.} Assume $(0,0)$ is a singular point.\\
Then 
\begin{equation}\label{e.blowupL^3}
\limsup_{t\rightarrow 0^+}\|v(\cdot,t)\|_{L^3(B_0(\sqrt{a})\setminus B_0(\sqrt{-at}))}=\infty.
\end{equation}
If in addition, the solution is axisymmetric, then
\begin{equation}\label{e.blowupL^3infty}
\limsup_{t\rightarrow 0^+}\|v(\cdot,t)\|_{L^{3,\infty}(B_0(\sqrt{a})\setminus B_0(\sqrt{-at}))}=\infty.
\end{equation}
Here, and throughout the paper, $a:=\lambda_S(B_0(1))>\pi$, where $\lambda_S(B_1)$ is the first eigenvalue of the Dirichlet-Stokes operator on $B_0 (1)$.\footnote{\label{foot.a}From our estimates, we can allow $a\in\big(0,\frac43\lambda_S(B_0(1))\big)$. We choose $a:=\lambda_S(B_0(1))$ in order to fix the ideas.}
\end{theorema}

\begin{figure}[h]
\begin{center}
\includegraphics[scale=.6]{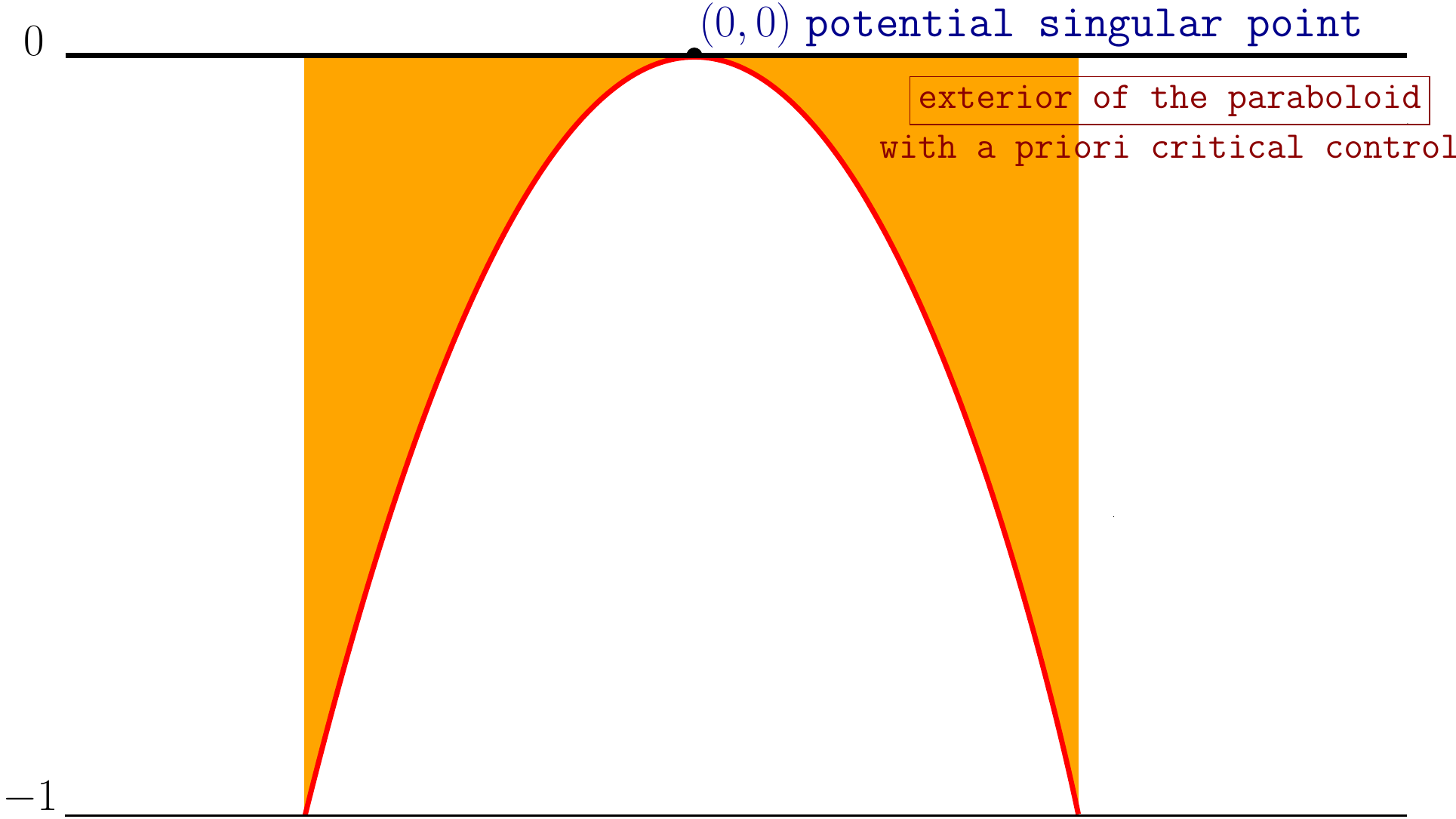}
\caption{Illustration for Theorem \ref{theom.main}}
\label{fig.paraboloid}
\end{center}
\end{figure}

Let us remark the following three facts:
\begin{enumerate}[label=(\arabic*)]
\item Theorem \ref{theom.main} is stated in the global setting for first-time singularities of Leray-Hopf solutions. In view of the global regularity problem for the three-dimensional Navier-Stokes equations this is an appropriate setting. That said, Theorem \ref{theom.main} can be localized i.e. proved for suitable solutions as is the result of Neustupa \cite[Theorem 1]{Neus14}. Localizing requires minor technical changes in the proof.
\item We believe that Theorem \ref{theom.main} can be quantified using the strategy developed in the works \cite{Tao19,BP21cmp,TB22,OP22} and further explained in the recent survey paper \cite{BP22survey}. The scale-invariant controls obtained in Section \ref{mainproofs} play a key role for this quantification as do the scale-invariant control $L^\infty_tL^3_x$ in \cite{Tao19,TB22} or the Type I assumption in \cite{BP21cmp}.
\item After finishing the paper, we were made aware of the work \cite{lei2011liouville} dealing with similar hollowed critical conditions in the axisymmetric case. There, De Giorgi type energy estimates and Moser type iteration are directly used for the equation for the swirl, which is pressureless, in order to deduce H\"older continuity.
\end{enumerate}

\subsection{Heuristics for the proof}

The proof is by contraposition, assuming that one of the quantities in \eqref{e.blowupL^3} or \eqref{e.blowupL^3infty} is bounded. To fix the ideas, let us assume that 
\begin{equation}\label{e.prooftheommainL3}
\esssup_{s\in(-1,0)}\|v(\cdot,s)\|_{L^{3}(B_0(\sqrt{a})\setminus B_0(\sqrt{a(-s)}))}\leq M,
\end{equation}
for some $M\geq 1$. 

\subsubsection*{Failure of a direct strategy} 
It seems not possible to directly prove that \eqref{e.prooftheommainL3} implies the re\-gularity by adapting the compactness argument of \cite{ESS2003}. Indeed, the absence of a global scale-invariant assumption prevents us from relying on persistence of singularity type results for obtaining a non-zero solution from a limiting procedure. It turns out that proving a scale-invariant bound inside the parabolic is the cornerstone of our strategy.

\subsubsection*{Key difficulty} 
We work in physical space variables rather than renormalized variables of Giga and Kohn type as is done by Ne\v{c}as and Neustupa \cite{NN02} and Neustupa \cite{Neus12,Neus14}. Our study is based on time-weighted scale-invariant local energy controls. The key point that enables results such as \eqref{e.blowupL^3} or \eqref{e.blowupL^3infty} to hold is the divergence-form structure of the non linear term $v\cdot\nabla v=\nabla\cdot(v\otimes v)$ in the Navier-Stokes equations. If the pressure was not there, it would be easy to obtain a scale-invariant control inside the paraboloid
\begin{equation}\label{e.paraboloid}
\bigcup_{s\in(-1,0)}B_0(\theta_a(s)),
\end{equation}
because the variation of the local energy localized on the parabolid 
only involves a priori control outside the paraboloid, for instance in the region 
$$
\bigcup_{s\in(-1,0)}B_0(2\theta_a(s))\setminus B_0(\theta_a(s)).
$$
The major difficulty we have to face is the pressure term in the local energy inequality
$$
\int\limits_{-1}^0\int\limits_{\R^3}pv\cdot\nabla\Psi^2\, dxd\tau.
$$
Indeed, because of the nonlocality of the pressure, bounding this term involves estimates of the velocity inside the paraboloid.

\subsubsection*{Key new idea} 
 Our idea to circumvent the problem of the pressure is to introduce a test function $\Psi_N$ that is supported in 
$$
\bigcup_{s\in(-1,0)}B_0((N+1)\theta_a(s))\setminus B_0(N\theta_a(s)),
$$
for $N\gg 1$ i.e. far away from the region \eqref{e.paraboloid} where we lack critical information. Taking $N$ large (in terms of $M$, see \eqref{e.prooftheommainL3}) enables us to show that the contribution of the problematic pressure term is negligible compared to an appropriate localized energy of $v$. This scheme is totally different from what is done in \cite{Neus14}. There, the fact that non-endpoint critical controls are considered buys some smallness that allows a control of the pressure term. Note that in \cite{Neus14} Neustupa eventually applies $\varepsilon$-regularity to conclude the regularity. Contrary to this, we do not have such smallness in the endpoint case and do not rely on $\varepsilon$-regularity, see Section \ref{mainproofs}.

\subsubsection*{Two technical points}

First, the constraint on $a$ in Theorem \ref{theom.main} (see also Footnote \ref{foot.a}) comes from the condition \eqref{e.choiceN} that is needed in order to get a damping term in the local energy inequality.

Second, as mentioned above, we work with scale-invariant weighted local energy quantities. The weights, see \eqref{ee2} below involve a parameter $\gamma$. In order for the local energy estimate \eqref{ee2} to be well-defined, we first need to take $\gamma>0$. This enables us, see Section \ref{subsec.scaleinv}, to prove that a certain scale-invariant kinetic energy is bounded (quantity $f$ in \eqref{deff}), which in turn enables us to make sense of \eqref{ee2} for $\gamma>-1$ and in particular for $\gamma=0$. This parameter is used in the final part of the paper, Section \ref{mainproofs}.

\subsection{Outline of the paper}

Section \ref{pressuresection} is concerned with scale-invariant weighted local energy estimates. The focus is on the control of the pressure term. Section \ref{sec.scaleinv} demonstrates the boundedness of a scale-invariant kinetic energy localized on the paraboloid \eqref{e.paraboloid}. The final part of the paper, Section \ref{mainproofs}, is devoted to the proof of Theorem \ref{theom.main}. Appendix \ref{sec.lorentz} surveys some results about Lorentz spaces.

\subsection{Notations}
\label{notations}
For $a>0$, we define $\theta_a (t) = \sqrt{a (- t)}$. Notice that throughout the paper $a:=\lambda_S(B_0(1))$ as stated in Theorem \ref{theom.main}. 

Throughout the paper $C$ is assumed to be a positive universal constant which may change from line to line. In particular $C$, does not depend on the parameters $a$, $\gamma$ or $N$ (see Section \ref{pressuresection}). Note that the inequality $\ldots\leq C(\ldots)$ is sometimes denoted by $\ldots\lesssim\ldots$. The notation $C_{b_1,\ldots\, b_k}$ denotes a positive constant depending on the parameters $b_1,\ldots\, b_k$. Note that the inequality $\ldots\leq C_{b_1,\ldots\, b_k}(\ldots)$ is sometimes denoted by $\ldots\lesssim_{b_1,\ldots\, b_k}\ldots$.

\section{Weighted energy and pressure estimates}
\label{pressuresection}

Let $M\geq 1$. In this section we assume
\begin{equation}
\label{leqm}
    \esssup_{ s \in (  - 1 , 0) }  \|v(\cdot,s)\|_{L^{3,\infty}(B_0(\sqrt{a})\setminus B_0(\theta_a(s)))}\leq M.
\end{equation}
For basic facts about weak Lebesgue and Lorentz spaces we refer to Appendix \ref{sec.lorentz}. We mainly rely on the H\"older inequality for Lorentz spaces in the estimates below, see Proposition \ref{hunt}. 
Our objective is to get scale-invariant controls for the velocity inside the region \eqref{e.paraboloid} 
via the local energy inequality and the use of the critical control \eqref{leqm}. The final scale-invariant control, see Proposition \ref{prop.scaleinv}, is obtained in Section \ref{sec.scaleinv}. Our focus in the present section is on an appropriate choice of test functions for the local energy inequality and on the estimates for the pressure.

Let\footnote{Later on, $N$ will be taken large depending in particular on $M$, see \eqref{e.choiceN}.} $N\geq4$ and $\gamma>0$.\footnote{The restriction $\gamma>0$ is needed for the validity of \eqref{ee1} for all $t\in(s,0]$. We will subsequently show, see \eqref{hypforene}, that the second term in the left hand side of \eqref{ee1} also makes sense for $\gamma>-1$. This fact will be used in Section \ref{mainproofs}.} 
Let $\varphi_N\in C^\infty_c(\R^3)$ with $\varphi_N$ positive, $\varphi_N(x)=1$ on $B_0(N)$, $\supp\varphi_N\subset B_0 (N+1)$ and $\|\nabla\varphi_N\|_{L^\infty}\lesssim 1$. We then test the Navier-Stokes equations \eqref{ns1intro} with 
\begin{equation}
    \label{testf}
\theta_a^\gamma(t) \Psi_N ^2 (x,t), \qquad t < 0,
\end{equation}
where 
\begin{equation}\label{e.eqPsi}
\Psi_N (x,t)  := \varphi_N \Big( \frac{x}{\theta_a(t)} \Big),
\end{equation}
and integrate over $\R^3\times(s,t)$ with $t\in(s,0]$.\footnote{For $t=0$, we integrate over $\R^3\times(s,\delta)$ and let $\delta\rightarrow 0^-$.}
This yields 
for all $ s\in(-1,0)$ and for all $t \in ( s, 0]$,
\begin{align}
\label{ee1}
    &\theta_a(t)^\gamma \| \Psi_N v ( \cdot , t ) \|_2 ^{2} + \frac{a \gamma }{2}  \int\limits_{s}^{t}  \theta_a(\tau)^{ \gamma - 2 }  \| \Psi_N v \|_2 ^{2} \, d\tau + 2 \int\limits_{s}^{t} \theta_a( \tau )^{ \gamma }  \| \Psi_N  \nabla  v \|_2^2 \, d\tau  \nonumber \\
    &\leq \theta_a(s) ^\gamma  \| \Psi_N v ( \cdot, s ) \|_2 ^{2} + \int\limits_{s}^{t}   \theta_a( \tau )^{ \gamma } \int\limits_{\R^3}  \big[-\nabla |v|^2 \cdot  \nabla  \Psi_N^2   + (|v|^2 + 2 p) (v \cdot \nabla \Psi_N^2) \big]  \, d \tau  \nonumber \\
     &\qquad+ \int\limits_{s}^{t}   \theta_a(\tau)^{ \gamma - 2 } \int\limits_{\R^3}   \Big( \frac{1}{2}a x \cdot \nabla  \Psi_N ^2 (x, \tau ) \Big) |v(x, \tau)|^2 \, d \tau .
\end{align}
Notice that 
\begin{equation}\label{e.boundxnabla}
\|a x \cdot \nabla  \Psi_N ^2\|_{L^\infty(\R^3)}\lesssim aN,
\end{equation}
which will be used below, see \eqref{e.controlKI}. 
Now, observe that the identity 
\begin{equation}\label{e.identity}
    \Psi_N^2 | \nabla v |^2 = | \nabla (\Psi_N v ) |^2 - | \nabla \Psi_N |^2 | v |^2 - \frac{1}{2} \nabla \Psi_N^2 \cdot \nabla | v |^2
\end{equation}
allows $\nabla (\Psi_N v) $ to appear. Hence, using the identity \eqref{e.identity} and multiplying \eqref{ee1} by $ \theta_a(s) ^{-\gamma -1}$, 
and for all $s\in(-1,0)$, $t \in (s, 0]$,
\begin{align}
\label{ee2}
    &\frac{ \theta_a (t)^\gamma }{ \theta_a(s) ^{\gamma +1} } \| \Psi_N v ( \cdot , t ) \|_2 ^{2} + \frac{a \gamma }{2}  \int\limits_{s}^{t} \frac{ \theta_a(\tau)^{ \gamma - 2 }  }{ \theta_a(s) ^{\gamma +1}  } \| \Psi_N v \|_2 ^{2}  \, d\tau  + 2\int\limits_{s}^{t}  \frac{ \theta_a (\tau)^\gamma } { \theta_a(s) ^{\gamma +1} } \| \nabla ( \Psi_N v ) \|_2^2  \, d\tau  \nonumber \\
    &\qquad\leq \frac{1}{ \theta_a(s) } \| \Psi_N v (\cdot, s) \|_2 ^{2} + \int\limits_{s}^{t}  \frac{  \theta_a (\tau)^\gamma }{ \theta_a(s) ^{\gamma +1} } \int\limits_{\R^3} \big[2|\nabla  \Psi_N|^2 |v|^2 + (|v|^2 + 2 p) (v \cdot \nabla \Psi_N^2)\big]  \, dx \, d \tau  \nonumber \\
    &\qquad\quad+ \int\limits_{s}^{t} \frac{ \theta_a(\tau)^{ \gamma - 2 } }{ \theta_a(s) ^{\gamma +1} } \int\limits_{\R^3} \frac{1}{2} a x \cdot \nabla  \Psi_N ^2   |v|^2 \, dx \, d \tau.
\end{align}
We now split the right hand side of \eqref{ee2} into the pressureless terms
\begin{align}\label{e.KI}
\begin{split}
    K^I(s,t) =& \int\limits_{s}^{t} \frac{ \theta_a(\tau)^\gamma }{\theta_a(s) ^ {\gamma + 1} } \int\limits_{\R^3} \big[ 2|\nabla  \Psi_N|^2 |v|^2 + |v|^2 v \cdot \nabla \Psi_N^2 \big] \, dx \, d \tau \\
    &+ \frac{1}{2} \int\limits_{s}^{t} \frac{ \theta_a(\tau)^{\gamma - 2 } }{\theta_a(s) ^ {\gamma + 1} } \int\limits_{\R^3}a x \cdot \nabla  \Psi_N ^2 (x, \tau)|v|^2 \, dx \, d\tau
\end{split}
\end{align}
and the pressure term
\begin{equation}\label{e.KII}
    K^{II}(s,t) =  \int\limits_{s}^{t} \frac{ \theta_a (\tau) ^\gamma }{\theta_a(s) ^{\gamma+1} } \int\limits_{\R^3} 2 p  (v \cdot \nabla \Psi_N^2 )\, dx \, d \tau.
\end{equation}
The quantity $K^I(s,t)$ is easy to bound using the a priori critical control \eqref{leqm}. Indeed, all the terms in the right hand side of \eqref{e.KI} involve a derivative of the cut-off function $\Psi_N$, hence, these integrals are all supported in the complement of \eqref{e.paraboloid}. The quantity $K^{II}(s,t)$ concentrates the difficulties, because of the nonlocality of the pressure. We handle this issue first.

\subsection{Pressure estimates}

To treat the quantity $K^{II}(s,t)$, observe that 
\begin{equation*}
    \nabla \Psi_N^2 (x, \tau) = \frac{2}{\theta_a(\tau)} \Psi_N(x,t) \nabla \varphi_N \Big( \frac{x}{\theta_a(\tau)} \Big),
\end{equation*}
thus,
\begin{equation*}
    \big|K^{II}(s,t)\big|  \leq   C_0 \int\limits_{s}^{t} \frac{ \theta_a(\tau)^{\gamma-1} }{\theta_a(s) ^ { \gamma + 1 } } \int\limits_{B_0((N+1)\theta_a (\tau))\setminus B_0(N\theta_a (\tau))}  |pv\Psi_N|  \, dx \, d \tau.
\end{equation*}
Then we decompose the pressure term.
Consider a  
cut-off function $\eta\in C^\infty_c(\R^3)$ such that:
\begin{equation*}
    \eta(x)= \left\{ \begin{array}{lcc}
             1 &   \text{if}  & |x| < \frac{1}{2} \sqrt{a},  \\
             \in [0,1] &  \text{if} & \frac{1}{2} \sqrt{a} < |x| < \frac{3}{4} \sqrt{a}, \\
             0 &  \text{if}  & \frac{3}{4} \sqrt{a} < |x| ,
             \end{array}
   \right.
\end{equation*}
and satisfying $ |\nabla \eta | < \frac{C}{\sqrt a}
$ and $|\nabla^2 \eta | <\frac{C}{a}$.
Let us also write $\eta(x) = \eta_1(x,t) + \eta_2(x,t)$, where $\eta_1$ and $\eta_2$ are $C^\infty_0$ with values in $[0,1]$, and such that $\eta_1(\cdot,t)$ is equal to $1$ on $ B_{0} (\theta_a(t)) $ and $\eta_1( \cdot ,t)$ is supported on $B_{0} (2 \theta_a (t) ) $.\footnote{Notice that this decomposition of $\eta$ is only needed for $t\in(-\frac14,0)$.} This decomposition is classical, see for instance \cite{CKN82}, except for the additional localization on the paraboloid via $\eta_1$. We summarize the decomposition in the following lemma.

\begin{Lemma}
\label{pd}
Let $v$ be a finite-energy weak solution to the Navier-Stokes equations \eqref{ns1intro} in $\R^3\times(-1,0)$ such that $v\in C^\infty(-1,T;C^\infty(\R^3))$ for all $T\in(-1,0)$. 
There exist a universal constant $C\in(0,\infty)$ such that for all $\tau \in ( -\frac{1}{16(N+1)^2 } , \, 0 )$ and for all $x \in B_0((N+1)\theta_a(\tau))\setminus B_0(N\theta_a(\tau))$,
\begin{align}
    |p(x,\tau)|=|\eta(x)p(x,\tau)| \leq & \frac{C}{ N^3 \theta_a(\tau)^3 }  \int\limits_{B_{0} (\theta_a (\tau) )} |v(y,\tau)|^2 \, dy \nonumber\\
 &+\frac{C}{ N^3 \theta_a(\tau)^3 }  \int\limits_{B_{0} (2 \theta_a (\tau) )\setminus B_{0} (\theta_a (\tau) ) } |v(y,\tau)|^2 \, dy \nonumber\\
    &+ \Bigg| \frac{1}{4 \pi} \int\limits_{ B_{0} ( \frac{3}{4} \sqrt{a} ) \setminus B_{0} (2\theta_a(\tau) ) }  \frac{\partial^2}{\partial_{y_i} \partial_{y_j}} \left(  \frac{1}{|x-y|} \right) [\eta_2 v_i v_j] (y,\tau ) \, dy \, \Bigg| \nonumber\\
    & + \frac{C}{(\sqrt{a} )^3} \int\limits_{ B_{0} ( \frac{3}{4} \sqrt{a} ) \setminus B_{0} ( \frac{1}{2} \sqrt{a} ) } [ |v|^2 + |p| ] (y,\tau) \, dy\nonumber\\
    =:&\bar p_1(\tau)+\bar p_2(\tau)+\bar p_3(x,\tau)+\bar p_4(\tau).\label{e.barpressure}
\end{align}
\end{Lemma}

\begin{proof}
Using the identity
\begin{equation*}
    \eta(x) p(x,\tau) =- \frac{1}{4 \pi} \int\limits_{\mathbb{R}^3} \frac{1}{|x-y|} [ \Delta (\eta p) ] (y,\tau) \, dy
\end{equation*}
and the fact that $\Delta p = -\partial_i \partial_j (v_i v_j) $, we obtain
\begin{equation*}
    \eta(x) p(x,\tau) := p_1(x,\tau)+ p_2(x,\tau)+ p_3(x,\tau)+p_4(x,\tau),
\end{equation*}
where
\begin{align*}
p_1(x,\tau) :=& \frac{1}{4 \pi} \int\limits_{B_{0} (\theta_a (\tau) ) } \frac{\partial^2}{\partial_{y_i} \partial_{y_j}} \left(  \frac{1}{|x-y|} \right) [\eta_1 v_i v_j] (y,\tau ) \, dy,
\end{align*}
\begin{align*}
   p_2(x,\tau) :=& \frac{1}{4 \pi} \int\limits_{B_{0} (2 \theta_a (\tau))\setminus B_{0} (\theta_a (\tau) )} \frac{\partial^2}{\partial_{y_i} \partial_{y_j}} \left(  \frac{1}{|x-y|} \right) [\eta v_i v_j] (y,\tau ) \, dy,
   \end{align*}
\begin{align*}
 p_3(x,\tau) :=& \frac{1}{4 \pi} \int\limits_{ B_{0} ( \frac{3}{4} \sqrt{a} ) \setminus B_{0} (2\theta_a(\tau) ) }  \frac{\partial^2}{\partial_{y_i} \partial_{y_j}} \left(  \frac{1}{|x-y|} \right) [\eta_2 v_i v_j] (y,\tau ) \, dy
\end{align*}
and
\begin{align*}
     p_4(x,\tau) :=&  \frac{1}{2 \pi} \int\limits_{ B_{0} ( \frac{3}{4} \sqrt{a} ) \setminus B_{0} ( \frac{1}{2} \sqrt{a} ) }   \frac{x_i - y_i}{|x-y|^3} \left(  \frac{\partial \eta}{ \partial_{y_j}}  v_i v_j \right)  (y,\tau ) \, dy \\
 &  + \frac{1}{4 \pi} \int\limits_{ B_{0} ( \frac{3}{4} \sqrt{a} ) \setminus B_{0} ( \frac{1}{2} \sqrt{a} ) }   \frac{1}{|x-y|} \left(  \frac{\partial^2 \eta}{\partial_{y_i} \partial_{y_j}}  v_i v_j \right)  (y,\tau ) \, dy \\
    & + \frac{1}{2 \pi} \int\limits_{ B_{0} ( \frac{3}{4} \sqrt{a} ) \setminus B_{0} ( \frac{1}{2} \sqrt{a} ) }   \frac{x_i - y_i}{|x-y|^3} \left(  \frac{\partial \eta}{ \partial_{y_j}}  \, p \right)  (y,\tau ) \, dy \\
     &  + \frac{1}{4 \pi} \int\limits_{ B_{0} ( \frac{3}{4} \sqrt{a} ) \setminus B_{0} ( \frac{1}{2} \sqrt{a} ) }   \frac{1}{|x-y|} \left(  \Delta \eta \, p \right)  (y,\tau ) \, dy.
\end{align*}
Notice that 
\begin{equation*}
   \dist\big(B_0(2 \theta_a (\tau)),B_0((N+1)\theta_a (\tau))\setminus B_0(N\theta_a (\tau))\big)\geq (N-2)\theta_a(\tau),
\end{equation*}
and 
\begin{equation*}
   \dist\big(B_0(\tfrac{1}{2}\sqrt{a}),B_0((N+1)\theta_a (\tau))\setminus B_0(N\theta_a (\tau))\big)\geq \frac{1}{4} \sqrt{a} 
\end{equation*}
for $\tau \in ( -  \frac{1}{16(N+1)^2}  , 0 )$. 
Combining the above two estimates with the above pressure decomposition concludes the proof of the lemma.\end{proof}

By Lemma \ref{pd}, we have for all $ s \in (  - \frac{1}{16(N+1)^2}, \, 0 )  $ and for all $ t \in (s,0] $,
\begin{multline}\label{e.estKII}
\big|K^{II} (s)\big| \\
\leq \frac{C} { \theta_a(s) ^{\gamma + 1}  } \int\limits_{s}^{t} \theta_a  (\tau)^{\gamma -1} \int\limits_{B_0((N+1)\theta_a (\tau))\setminus B_0(N\theta_a (\tau))}  \big( \bar p_1(\tau)  + \bar p_2(\tau)   +  \bar p_3(x,\tau)  + \bar p_4(\tau) \big) |v\Psi_N|  \, dx \, d \tau,
\end{multline}
where $\bar p_i$, $i=1,\ldots\, 4$ are defined in \eqref{e.barpressure}. We now turn to the estimates of each term in the right hand side of \eqref{e.estKII}.

\subsubsection*{Term involving $\bar p_1$}
The term involving $\bar p_1$ is the main source of difficulties in this paper. Indeed, this term concentrates the nonlocal part of the pressure that cannot be directly controlled via the a priori critical assumption. Here we control this term via the kinetic energy in $B_{0} (  \theta_a(\tau) )$. In Section \ref{gronestimate} we will then be able to control the scale-invariant local energy via a Gronwall type argument. 

From H\"older's inequality for Lorentz spaces, see Proposition \ref{hunt}, the fact that
\begin{equation*}\label{e.L321}
\|\mathbf{1}_{B_0((N+1)\theta_a (\tau))\setminus B_0(N\theta_a (\tau))}\|_{L^{\frac32,1}(\R^3)}\leq CN^2\theta_a(\tau)^2,
\end{equation*}
and by \eqref{leqm}, we obtain
\begin{align}\label{e.boundvL1}
\begin{split}
    \int\limits_{B_0((N+1)\theta_a (\tau))\setminus B_0(N\theta_a (\tau))} | v (x,\tau)| \, dx \leq& C N^2 \theta_a(\tau)^2 \|v(\cdot,\tau)\|_{L^{3,\infty}(B_0((N+1)\theta_a (\tau))\setminus B_0(N\theta_a (\tau)))}\\
    \leq&C N^2M \theta_a(\tau)^2.
\end{split}
\end{align}
Therefore
\begin{align}
    &\int\limits_{s}^{t} \theta_a(\tau) ^{\gamma -1}  \int\limits_{B_0((N+1)\theta_a (\tau))\setminus B_0(N\theta_a (\tau))}  \bar p_1(\tau) |v(x,\tau)\Psi_N|    \, dx \, d \tau\nonumber \\
    &\leq  \, \frac{C}{N^3} \int\limits_{s}^{t} \theta_a(\tau) ^{\gamma-4}  \Bigg(  \int\limits_{ B_{0} (  \theta_a(\tau) ) } |v(x,\tau)|^2 \, dx \Bigg) \int\limits_{B_0((N+1)\theta_a (\tau))\setminus B_0(N\theta_a (\tau))} | v(x,\tau)| \, dx 
    \, d \tau\nonumber \\ 
    &\leq  \, \frac{C_*M}{N} 
     \int\limits_{ s}^{t} \theta_a(\tau) ^{\gamma -2}  \int\limits_{ B_{0} (  \theta_a(\tau) ) } |v(x,\tau)|^2 \, dx \, d\tau,\label{e.definitionC*}
\end{align}
where $C_*\in(0,\infty)$ is a universal constant.

\subsubsection*{Term involving $\bar p_2$}

We estimate this term using H\"older's inequality for Lorentz spaces, \eqref{leqm} and \eqref{e.boundvL1},
\begin{align*}
    &\int\limits_{s}^{t} \theta_a(\tau) ^{\gamma -1}  \int\limits_{B_0((N+1)\theta_a (\tau))\setminus B_0(N\theta_a (\tau))}  \bar p_2(\tau) |v(x,\tau)\Psi_N|    \, dx \, d \tau \\
     &\leq \, \frac{C}{N^3} \int\limits_{s}^{t} \theta_a(\tau) ^{\gamma-4}  \Bigg( \int\limits_{ B_{0} ( 2 \theta_a(\tau) )  \setminus B_{0} (\theta_a(\tau) ) } |v(x,\tau)|^2 \, dy \Bigg) \int\limits_{B_0((N+1)\theta_a (\tau))\setminus B_0(N\theta_a (\tau))} | v(x,\tau)|  \, dx   \, d \tau \\
     &\leq \, \frac{CM^3}{N}\int\limits_{s}^{t} \theta_a(\tau) ^{\gamma -1}  \, d\tau\\
 &\leq \frac{CM^3}{a(\gamma + 1)N} \theta_a(s)^{ \gamma + 1 }.
\end{align*}

\subsubsection*{Term involving $\bar p_3$}

Observe that for this term, the singularity of the kernel is seen. Calder\'on-Zygmund's theorem gives for $\tau \in ( -\frac{1}{16(N+1)^2 } , \, 0 )$,
\begin{equation}
\label{CZ}
\|\bar p_3 (\cdot,\tau)\|_{L^{\frac32,\infty}(B_0((N+1)\theta_a (\tau))\setminus B_0(N\theta_a (\tau)))}\leq C\|v(\cdot,\tau)\|_{L^{3,\infty}(B_{0} ( \frac{3}{4}\sqrt{a} ) \setminus B_{0} (\theta_a(\tau)))}^2\leq CM^2.
\end{equation}
Then, by H\"older's inequality for Lorentz spaces, interpolation of $L^{3,1}$ between $L^2$ and $L^6$ \cite[Theorem 5.3.1]{BL76}, the Sobolev inequality and Young's inequality,
\begin{align*}
    &\int\limits_{s}^{t} \theta_a(\tau) ^{\gamma -1}  \int\limits_{B_0((N+1)\theta_a (\tau))\setminus B_0(N\theta_a (\tau))} \bar p_3(x,\tau) |v(x,\tau)\Psi_N|    \, dx \, d \tau \\
    &\leq  \int\limits_{s}^{t} \theta_a(\tau) ^{\gamma -1} \|\bar p_3 (\cdot,\tau)\|_{L^{\frac32,\infty}(B_0((N+1)\theta_a (\tau))\setminus B_0(N\theta_a (\tau)))} \|v(\cdot,\tau)\Psi_N\|_{L^{3,1}(\R^3)}\, d \tau \\
    &\leq C_{\bar p_3}  M^2\int\limits_{s}^{t}\theta_a(\tau) ^{\gamma-\frac32}  \|\Psi_Nv(\cdot,\tau)\|_{L^2(\R^3)}\, d \tau+C_{\bar p_3}  M^2\int\limits_{s}^{t}\theta_a(\tau) ^{\gamma-\frac12}  \|\nabla (\Psi_Nv(\cdot,\tau))\|_{L^2(\R^3)}\, d \tau,
\end{align*}
where $C_{\bar p_3}\in(0,\infty)$ is a universal constant. 
Now, by the Cauchy-Schwarz and Young's inequalities we have on the one hand
\begin{align*}
&C_{\bar p_3}M^2\int\limits_{s}^{t}\theta_a(\tau) ^{\gamma-\frac32}  \, \|\Psi_Nv(\cdot,\tau)\|_{L^2(\R^3)}\\
&\leq C_{a,\gamma}M^4\theta_a(s)^{\gamma+1}+\frac{a\gamma}{16}\int\limits_s^t\theta_a(\tau)^{\gamma-2}\|\Psi_Nv(\cdot,\tau)\|_{L^2}^2\, d\tau
\end{align*}
and on the other hand
\begin{align*}
&C_{\bar p_3}M^2\int\limits_{s}^{t}\theta_a(\tau) ^{\gamma-\frac12}  \, \|\nabla(\Psi_Nv(\cdot,\tau))\|_{L^2(\R^3)}\\
&\leq C_{a,\gamma}M^4\theta_a(s)^{\gamma+1}+\frac{1}{4}\int\limits_s^t\theta_a(\tau)^{\gamma}\|\nabla(\Psi_Nv(\cdot,\tau))\|_{L^2}^2\, d\tau.
\end{align*}

\subsubsection*{Term involving $\bar p_4$}

We have by H\"older's inequality for Lorentz spaces and \eqref{e.boundvL1}
\begin{align*}
    &\int\limits_{s}^{t} \theta_a(\tau)^{\gamma -1} \int\limits_{B_0((N+1)\theta_a (\tau))\setminus B_0(N\theta_a (\tau))} \bar p_4(\tau) |v(x,\tau)\Psi_N|    \, dx \, d \tau \\
    &\leq  \, \frac{CN^2M}{(\sqrt{a} )^3} \int\limits_{s}^{t} \theta_a(\tau) ^{\gamma +1} \Bigg( \int\limits_{ B_{0} ( \frac{3}{4}\sqrt{a} ) \setminus B_{0} \left( \frac{1}{2}\sqrt{a} \right)  } \big[ |v|^2 + |p| \big] \, dx \Bigg)\, d \tau \\
    &\leq  \, \frac{CSMN^2}{(\sqrt{a} )^3}\theta_a(s) ^{\gamma + 1 },
    \end{align*}
where
\begin{equation*}
    S = \int\limits_{ - 1}^{0} \Bigg( \int\limits_{ B_{0} ( \frac{3}{4}\sqrt{a} ) \setminus B_{0} ( \frac{1}{2}\sqrt{a} )  } \big[ |v|^3 + |p|^{\frac{3}{2}} \big] \, dx \Bigg)^{ \frac{2}{3} } <\infty.
\end{equation*}

\subsection{Local energy estimates under a priori critical control outside of the paraboloid}

We recall that $M\geq 1$ and $N\geq 4$. 
Summarizing the pressure estimates above 
we get, for all $s \in (  - \frac{1}{16(N+1)^2 } , \, 0 ) $ and for all $t \in (s,0]$, first $\theta_a(s)\leq\frac{\sqrt a}{4(N+1)}$ and then
\begin{multline}\label{e.controlKII}
    \big|K^{II}(s)\big| \leq \, \frac{CM }{N} \int\limits_{s}^{t}  \frac{ \theta_a(\tau)^{ \gamma - 2 }  }{ \theta_a(s) ^{\gamma +1}  } \| \Psi_N v \|_2 ^{2}\, d\tau +\frac{a\gamma}{16}\int\limits_s^t\frac{ \theta_a(\tau)^{ \gamma - 2 }  }{ \theta_a(s) ^{\gamma +1}  }\|\Psi_Nv(\cdot,\tau)\|_{L^2}^2\, d\tau\\
    + \frac{1}{4}\int\limits_s^t\frac{ \theta_a(\tau)^{ \gamma}  }{ \theta_a(s) ^{\gamma +1}  }\|\nabla(\Psi_Nv(\cdot,\tau))\|_{L^2}^2\, d\tau+C_{a,\gamma}(M^4+SMN^2).
    \end{multline}
Similarly, for $ K^{I}(s)$, using in particular \eqref{e.boundxnabla}, we get
\begin{multline}\label{e.controlKI}
    \big|K^{I}(s)\big| \leq \frac{a\gamma}{16}\int\limits_s^t\frac{ \theta_a(\tau)^{ \gamma-2}  }{ \theta_a(s) ^{\gamma +1}  }\|\Psi_Nv(\cdot,\tau)\|_{L^2}^2\, d\tau+ \frac{1}{4}\int\limits_s^t\frac{ \theta_a(\tau)^{ \gamma}  }{ \theta_a(s) ^{\gamma +1}  }\|\nabla(\Psi_Nv(\cdot,\tau))\|_{L^2}^2\, d\tau\\
    +C_{a,\gamma}(M^4+M^2N^2).
    \end{multline}
Then, from \eqref{ee2} combined with \eqref{e.controlKI} and \eqref{e.controlKII}, we find for $s \in (  - \frac{1}{16(N+1)^2 } , 0 ) $, and $t \in (s,0]$,
\begin{align}
\label{ee3}
    \frac{ \theta_a (t)^\gamma }{\theta_a(s) ^{\gamma +1} } & \| \Psi_N v (\cdot,t) \|_2 ^{2} + \frac{a \gamma}{4} \int\limits_{s}^{t}  \frac{ \theta_a(\tau)^{ \gamma - 2 }  }{ \theta_a(s) ^{\gamma +1}  } \| \Psi_N v \|_2 ^{2}\, d\tau + \int\limits_{s}^{t}  \frac{ \theta_a (\tau)^\gamma } { \theta_a(s) ^{\gamma +1} } \| \nabla ( \Psi_N v ) \|_2^2\, d\tau  \nonumber\\
    \leq &  \frac{1}{\theta_a(s) } \| \Psi_N v (\cdot,s) \|_2 ^{2} +  \frac{C_*M}{N} \int\limits_{s}^{t}  \frac{ \theta_a(\tau)^{ \gamma - 2 }  }{ \theta_a(s) ^{\gamma +1}  } \| \Psi_N v \|_2 ^{2}\, d\tau
     \nonumber \\
    & +C_{a,\gamma}(M^4+SMN^2+M^2N^2).
    \end{align}
where $C_*\in(0,\infty)$ is the universal constant in \eqref{e.definitionC*}.

Let us emphasize that the gain of the factor $N^{-1}$ in the term $\frac{C_*M}{N} \int\limits_{s}^{t}  \frac{ \theta_a(\tau)^{ \gamma - 2 }  }{ \theta_a(s) ^{\gamma +1}  } \| \Psi_N v \|_2 ^{2}$ plays a crucial role below. Taking $N$ large makes it possible to make $\frac{C_*M}{N}$ small; see \eqref{e.choiceN}.

\section{Scale-invariant bounds on the interior of the paraboloid}
\label{sec.scaleinv}

The objective of this section is to prove the following result. 

\begin{Proposition}\label{prop.scaleinv}
Let $M\geq 1$. Let $v$ be a finite-energy weak solution to the Navier-Stokes equations \eqref{ns1intro} in $\R^3\times(-1,0)$ such that $v\in C^\infty(-1,T;C^\infty(\R^3))$ for all $T\in(-1,0)$.\\
Assume that
\begin{equation*}
    \esssup_{ s \in (  - 1 , 0) }  \|v(\cdot,s)\|_{L^{3,\infty}(B_0(\sqrt{a})\setminus B_0(\theta_a(s)))}\leq M.
\end{equation*}
Then,
\begin{equation}\label{controldef}
    \esssup_{s \in \left(  -1,0 \right) }  \frac{1}{\theta_a(s)}\int\limits_{B_0(\theta_a(s))}|v(x,\tau)|^2\, d\tau< + \infty.
\end{equation}
\end{Proposition}
 
We shall connect information between critical quantities in order to rewrite the time weighted energy estimates as a differential inequality.
Let us denote
\begin{equation}
    \label{deff}
    f(s) :=  \frac{1}{ \theta_a(s) } \| \Psi_N v(\cdot,s) \|_2^2 
\end{equation}
and for all $\gamma\in\R$ we define\footnote{Notice that $f$ and $g_\gamma$ are dimensionally critical in the sense of Caffarelli, Kohn and Nirenberg \cite{CKN82}.}
\begin{equation}
    \label{defg}
    g_\gamma(s, t) := 
    \int\limits_{s}^{t} \frac{ \theta_a(\tau)^{ \gamma - 1 } } { \theta_a(s) ^{\gamma +1} }    f(\tau) \, d \tau .
\end{equation}
We compute the $s$ derivative of $g_\gamma(\cdot,t)$,
\begin{align}
\label{lfg}
    \frac{\partial g_\gamma}{\partial  s}(s,t) &= \frac{a(\gamma +1)}{2} \theta_a(s) ^{-(\gamma + 3)} \int\limits_{s}^{t} \theta_a(\tau)^{ \gamma - 2 }   \| \Psi_N v (\cdot,\tau) \|_2 ^{2} \, d \tau - \theta_a(s) ^{-(\gamma +1)} \theta_a(s)^{ \gamma - 1 }   f(s) \nonumber \\
    &= \theta_a(s)^{-2}\left[ \frac{a(\gamma +1)}{2} g_\gamma(s,t) - f(s) \right].
\end{align}
Now, let us rewrite \eqref{ee3} so as to make the functions $f$ and $g_\gamma$ appear,
\begin{align}
\begin{split}
\label{ee3m}
    \frac{ \theta_a (t)^ {\gamma+1} }{ \theta_a (s) ^{\gamma +1} } & f(t) + \frac{a \gamma}{4} g_\gamma (s,t) + \int\limits_{s}^{t}  \frac{ \theta_a (\tau)^\gamma } { \theta_a (s) ^{\gamma +1} } \| \nabla ( \Psi_N v ) \|_2^2 \, d \tau   \\
    \leq &  \, f(s) + \frac{C_* M}{N}g_\gamma(s,t)+ C_{a,\gamma}(M^4+SMN^2+M^2N^2).
\end{split}
\end{align}
Finally, we replace the identity \eqref{lfg} in \eqref{ee3m} to get for $s \in (  - \frac{1}{16(N+1)^2 } , 0 ) $, and $t \in (s,0 ]$,
\begin{align}
\begin{split}
\label{ee4}
     \frac{ \theta_a (t)^ {\gamma+1} }{ \theta_a(s) ^{\gamma +1} } f(t) \, + \, & \theta_a(s)^{2} \frac{\partial g_\gamma}{\partial s} (s,t) - \Big( \frac{a}{2}\Big(\frac\gamma 2+1\Big) + \frac{C_*M}{N}\Big) g_\gamma(s,t) + \int\limits_{s}^{t}  \frac{ \theta_a (\tau)^\gamma } { \theta_a (s) ^{\gamma +1} } \| \nabla ( \Psi_N v ) \|_2^2 \, d \tau   \\
     \leq& C_{a,\gamma}(M^4+SMN^2+M^2N^2).
\end{split}
\end{align}
where $C_*\in(0,\infty)$ is the universal constant in \eqref{e.definitionC*}.

\subsection{Gaining control through a Friedrich's type estimate}
\label{gain}

We demonstrate the following Poincar\'e-type inequality
\begin{equation*}
    \int\limits_{s}^{t}  \frac{ \theta_a (\tau)^\gamma } { \theta_a (s) ^{\gamma +1} } \| \nabla ( \Psi_N v ) \|_2^2 \, d \tau
    \geq \, \Lambda g_\gamma(s,t) - C_\Lambda NM^2
\end{equation*}
for some $\Lambda> \frac{a}{2}$ and a constant $C _\Lambda\in(0,\infty)$ depending on $\Lambda$; see \eqref{gaingrad} below. For that, we adapt the computations made in \cite{Neus14} by Neustupa to the weighted energy norms used in our work.

\subsubsection*{Step 1: a partition of the test function $\varphi_N$}

Consider $\xi \in (0, 1)$ and a test function $\varphi_{1}^\xi$ defined as
\begin{equation*}
    \varphi_{1}^\xi (x)= \left\{ \begin{array}{lcc}
             1 &   \text{if}  & |x|< 1 + \frac{1}{4} \xi  \\
             \in [0,1] &  \text{if} & 1 + \frac{1}{4} \xi < |x| < 1 + \frac{3}{4} \xi \\
             0 &  \text{if}  & 1 + \frac{3}{4} \xi  < |x|
             \end{array}
   \right.
\end{equation*}
and satisfying 
\begin{equation}
\label{consxi}
    |\nabla \varphi_{1}^\xi | \leq 4 \xi^{-1}.
\end{equation}
Let us define $\varphi_{N,2}^\xi := \varphi_N - \varphi_{1}^\xi$. Moreover, we denote $\Psi_{1}^\xi (x,t)  := \varphi_{1}^\xi (\frac{x}{\theta_a(t)}) $ and $\Psi_{N,2}^\xi :=\varphi_{N,2}^\xi (\frac{x}{\theta_a(t)}) $. Thus, $ \Psi_N = \Psi_{1}^\xi + \Psi_{N,2}^\xi $ in view of \eqref{e.eqPsi}.

Observe that
\begin{align*}
    \| \nabla (\Psi_{1}^\xi v) \|_2^2 &= \|  \Psi_{1}^\xi \nabla v \|_2^2 + 2 \big( \Psi_{1}^\xi \nabla v ,\nabla (\Psi_{1}^\xi ) \otimes  v \big)_{L^2} + \|  \nabla (\Psi_{1}^{\xi} ) \otimes  v \|_2^2 \\
    &\leq \|  \nabla v \|_{L^2 (\supp(\Psi_{1}^\xi) ) }^2 + \int\limits_{\R^3} \partial_i \big(({\Psi_{1}^\xi})^2\big) v_j  \partial_i v_j  \, dx + \|  \nabla (\Psi_{1}^{\xi} ) \otimes  v \|_2^2 \\
    &\leq \|  \nabla ( \Psi_N v ) \|_{2}^2 -\frac{1}{2} \int\limits_{\R^3} \Delta (\Psi_{1}^\xi) | v |^2 \, dx + \|  \nabla (\Psi_{1}^{\xi} ) \otimes  v \|_2^2
\end{align*}
Here we used that $\Psi_N=1$ on the support of $\Psi_{1}^{\xi}$. 
Thus,
\begin{equation}
\label{gradpsi1}
     \int\limits_{s}^{t}  \frac{ \theta_a (\tau)^\gamma } { \theta_a (s) ^{\gamma +1} } \| \nabla ( \Psi_N v ) \|_2^2 \, d \tau \geq \int\limits_{s}^{t}  \frac{ \theta_a (\tau)^\gamma } { \theta_a (s) ^{\gamma +1} } \| \nabla ( \Psi_{1}^\xi v ) \|_2^2 \, d \tau  - c_{1}(s,t,a,\xi),
\end{equation}
where 
\begin{equation*}
    c_{1}(s,t,a,\xi) := \int\limits_{s}^{t} \frac{ \theta_a (\tau)^\gamma } { \theta_a (s) ^{\gamma +1} } \Bigg( - \frac{1}{2} \int\limits_{\supp (\nabla \Psi_{1}^\xi )} \Delta (\Psi_{1}^\xi) | v |^2 \, dx + \|  \nabla (\Psi_{1}^{\xi} ) \otimes  v \|_2^2  \Bigg) \, d \tau .
\end{equation*}
Using the fact that
\begin{equation*}
    \nabla \Psi_{1}^\xi = \frac{1}{\theta_a(t)} \nabla \varphi_{1}^\xi \Big( \frac{x}{\theta_a(t)} \Big),
\end{equation*}
we get
\begin{equation*}
    |c_{1}(s,t,a,\xi)| \leq \frac{C}{\xi^2} \int\limits_{s}^{t} \frac{ \theta_a (\tau)^{\gamma-2} } { \theta_a (s) ^{\gamma +1} }  \int\limits_{\supp (\nabla \Psi_{1}^\xi )} | v (x, \tau)|^2 \, dx   \, d \tau ,
\end{equation*}
and then using H\"older's inequality for Lorentz spaces, we obtain
\begin{equation*}
    - c_{1}(s,t, \xi) \geq -C_{a,a,\gamma,\xi }M^2,
\end{equation*}
where $C_{a,\gamma,\xi }\in(0,\infty)$ is a constant depending on $a$, $\gamma$ and $\xi$. 

\subsubsection*{Step 2: a Poincar\'e inequality}
Now, we need to manipulate the quantity $ \| \nabla ( \Psi_{1}^\xi v ) (\cdot, \tau) \|_2^2 $. 
Consider $  \tau \in (  - \frac{1}{16(N+1)^2} , 0 ) $. 
Since 
\begin{equation*}
    \int\limits_{B_0((1+\xi)\theta_a(\tau))} \nabla ( \Psi_{1}^\xi ) \cdot  v =0,
\end{equation*}
we can use a right inverse of the divergence\footnote{Such an operator is sometimes called a Bogovskii operator, see \cite{Galdibook}.} 
to infer that there exists a function 
$$w^\xi (\cdot,\tau)   \in W_0^{1,2} \big( B_0((1+\xi)\theta_a(\tau))\big)$$ 
such that
\begin{equation*}
    \nabla \cdot w^\xi (\cdot,\tau) = \nabla ( \Psi_{1}^\xi ) \cdot v (\cdot,\tau)
\end{equation*}
and
\begin{equation*}
    \| \nabla w^\xi \|_{L^2 \left( B_0((1+\xi)\theta_a(\tau))\right) } \leq C \| \nabla ( \Psi_{1}^\xi ) \cdot v \|_2,
\end{equation*}
where $C\in(0,\infty)$ does not depend on the parameters $\xi,\tau$. 
Hence $ \Psi_{1}^\xi  v - w^{\xi}  $ is divergence free in $B_{0} {((1+\xi)\theta_a(\tau))}$ with zero trace on the boundary. By Poincar\'e's inequality for trace-free divergence-free functions, we have
\begin{equation}\label{e.poincareDS}
    \| \Psi_{1}^\xi  v - w^\xi \|_{L^2 ( B_{0} { ((1+\xi)\theta_a(\tau) )   } ) } \leq \frac{1}{ \sqrt{ \lambda_S (B_0 {((1+\xi)\theta_a(\tau))}) }  } \| \nabla ( \Psi_{1}^\xi  v - w^\xi ) \|_{L^2 ( B_{0} { ((1+\xi)\theta_a(\tau)) }) },
\end{equation}
where $\lambda_S(B_0 {( r   ) })$ is the first eigenvalue of the Dirichlet-Stokes operator on $B_0 (r)$.\footnote{It is well known that $\lambda_S(B_0 (1))$ is greater than $\pi^2$, the principal eigenvalue of the Dirichlet-Laplace operator.} 
By homogeneity we find
\begin{equation}\label{e.scalingpoinca}
    \frac{1}{ \sqrt{ \lambda_S(B_0 ((1+\xi)\theta_a (\tau) )) }  } =  \frac{ (1+ \xi) \theta_a (\tau) }{ \sqrt{ \lambda_S (B_0{(1)}) } }.
\end{equation}
Thus, using Poincar\'e's inequality \eqref{e.poincareDS}, the scaling property \eqref{e.scalingpoinca} and \eqref{consxi} we find 
\begin{align*}
    \|  \Psi_{1}^\xi  v \|_2 &\leq \| \Psi_{1}^\xi  v - w^\xi \|_{L^2 (B_{0} {((1+\xi)\theta_a(\tau)) }) } +  \|  w^\xi \|_2 \\
    &\leq \frac{ (1+ \xi) \theta_a (\tau) }{ \sqrt{ \lambda_S(B_0(1)) } } \| \nabla ( \Psi_{1}^\xi  v - w^\xi ) \|_{L^2 (B_{0} { ((1+\xi)\theta_a(\tau)) }) } + \frac{ (1+ \xi) \theta_a (\tau) }{ \pi }  \|  \nabla w^\xi \|_{L^2 (B_{0} { ( (1+\xi)\theta_a(\tau)) }) } \\
    &\leq \frac{ (1+ \xi) \theta_a (\tau) }{ \sqrt{ \lambda_S(B_0(1)) } } \| \nabla ( \Psi_{1}^\xi  v  ) \|_{L^2 (B_{0} { ( (1+\xi)\theta_a(\tau)) }) } + \frac{ C (1+ \xi) }{ \xi \pi }  \|  v \|_{L^2 (B_{0} { ( (1+\xi)\theta_a(\tau))})\setminus B_{0} { (\theta_a(\tau))})}  . 
\end{align*}
As we want to manipulate the square of $\| \nabla ( \Psi_{1}^\xi  v  ) \|_{L^2 ( B_{0} {(1+\xi)\theta_a(\tau) } ) }$, we introduce a parameter $\kappa >0$ to get
\begin{equation*}
    \|  \Psi_{1}^\xi  v \|_2^2 \leq \frac{ (1+ \xi)^2(1 + \kappa ) \theta_a(\tau)^2 }{ \lambda_S(B_0(1)) }  \| \nabla ( \Psi_{1}^\xi  v  ) \|^2_{L^2 (B_{0} { ((1+\xi)\theta_a(\tau)} ) } + C_{ \xi, \kappa }  \|  v \|_{L^2 (B_{0} { ( (1+\xi)\theta_a(\tau))})\setminus B_{0} { (\theta_a(\tau))})}^2 ,
\end{equation*}
where $ C_{ \xi, \kappa }\in(0,\infty)$ is a constant depending on $\xi$ and $\kappa$.

Substituting this into \eqref{gradpsi1} gives that for $s \in (  - \frac{1}{16(N+1)^2} , 0 ) $ and $ t\in (s,0]$,
\begin{align}
\label{psi1}
     \int\limits_{s}^{t}  \frac{ \theta_a (\tau)^\gamma } { \theta_a (s) ^{\gamma +1} } \| \nabla ( \Psi_N v ) \|_2^2 \, d \tau \geq \frac{\lambda_S(B_0(1))}{(1 + \kappa)(1 + \xi)^2  }  \int\limits_{s}^{t}  \frac{ \theta_a (\tau)^ {\gamma-2} } { \theta_a (s) ^{\gamma +1} } \|  \Psi_{1}^\xi v  \|_2^2 \, d\tau - c_{2}(s,t,a,\xi,\kappa),
\end{align}
where,
\begin{equation*}
    c_{2}(s,t,a,\xi,\kappa) = \frac{C_{ \xi,\kappa }\lambda_S(B_0(1))}{(1 + \kappa)(1 + \xi)^2  }  \int\limits_{s}^{t}  \frac{ \theta_a (\tau)^ {\gamma-2} } { \theta_a (s) ^{\gamma +1} } \| v  \|_{L^2 \left(B_0((1+\xi)\theta_a(\tau))\setminus B_0(\theta_a(\tau))\right) }^2 \, d\tau  + c_{1} (s,t,a,\xi).
\end{equation*}
By the same reasoning as for $c_1$ 
\begin{equation*}
    -c_{2}(s,t,a,\xi, \kappa) \geq -C_{a,\gamma, \xi, \kappa}M^2,
\end{equation*}
where $C_{a,\gamma,\xi, \kappa }\in(0,\infty)$ is a constant depending on $a$, $\gamma$, $\xi$ and $\kappa$. 
Now, we use the identity
\begin{equation*}
    \| \Psi_{1}^\xi v \|_2^2 = \|  \Psi_N v \|_2^2 - 2 \big( \Psi_{N}^\xi v , \Psi_{N,2}^\xi v \big)_{L^2} + \|   \Psi_{N,2}^{\xi}  v \|_2^2
\end{equation*}
and \eqref{psi1} to get
\begin{align}
\label{psi}
     \int\limits_{s}^{t} \frac{ \theta_a (\tau)^ {\gamma} } { \theta_a (s) ^{\gamma +1} } \| \nabla ( \Psi_N v ) \|_2^2 \, d \tau  \geq \frac{\lambda_S(B_0 (1))}{(1 + \kappa)(1 + \xi)^2  }  \int\limits_{s}^{t}  \frac{ \theta_a (\tau)^ {\gamma-2} } { \theta_a (s) ^{\gamma +1} } \|  \Psi_N v  \|_2^2 \, d \tau - c_{3}(s,t,a,\xi,\kappa),
\end{align}
where,
\begin{align*}
    c_{3}(s,t,a,\xi,\kappa) =  \frac{\lambda_S(B_0 (1))}{(1 + \kappa)(1 + \xi)^2  } \int\limits_{s}^{t}   \frac{ \theta_a (\tau)^ {\gamma-2} } { \theta_a (s) ^{\gamma +1} } \left[  2 (\Psi_{N}^\xi v, \Psi_{N,2}^\xi v )_2-\| \Psi_{N,2}^\xi v  \|_{2}^2 \right] d\tau+c_{2} (s,t,a,\xi, \kappa).
\end{align*}
As $\Psi_{N,2}^{\xi} (\cdot, \tau)$  is supported on $B((N+1)\theta_a(\tau))\setminus B(\theta_a(\tau))$, we have by Hunt's inequality
\begin{equation*}
    -c_{3}(s,t,a,\xi, \kappa) \geq - C_{a,\gamma, \xi, \kappa }NM^2,
    \end{equation*}
where the constant $C_{a,\gamma, \xi, \kappa}\in(0,\infty)$ depends on $a$, $\gamma$, $\xi$ and $\kappa$. 
We finally find from \eqref{psi}, for $s \in (  - \frac{1}{16(N+1)^2} , 0 ) $ and $t \in (s,0 ]$,
\begin{equation}
\label{gaingrad}
    \int\limits_{s}^{t}   \frac{ \theta_a (\tau)^ {\gamma} } { \theta_a (s) ^{\gamma +1} } \| \nabla ( \Psi_N v ) \|_2^2 \, d \tau 
    \geq \,  \frac{\lambda_S(B_0(1))}{(1 + \kappa)(1 + \xi)^2  } \, g_\gamma(s,t) - C_{a,\gamma, \xi, \kappa}N M^2,
\end{equation}
where the constant $C_{a,\gamma, \xi, \kappa}\in(0,\infty)$ depends on $a$, $\gamma$, $\xi$ and $\kappa$. 

\subsection{Gronwall estimate}
\label{gronestimate}
With the Poincar\'e-type inequality obtained in \eqref{gaingrad}, the estimate \eqref{ee4} becomes for $s\in(-\frac1{16(N+1)^2},0)$, and $t\in(s,0]$,
\begin{align*}
     \frac{ \theta_a (t)^ {\gamma+1} }{ \theta_a(s) ^{\gamma +1} } f(t) \, + \, & \theta_a(s)^{2} \frac{\partial g_\gamma}{\partial s} (s,t)  + \left( \frac{\lambda_S(B_0(1))}{(1+ \kappa)(1 + \xi )^2} - \frac{a}{2}\big(\frac\gamma 2+1\big) - \frac{C_*M}{N}\right) g_\gamma (s,t)  \nonumber \\
    \leq & C_{a,\gamma}(M^4+SMN^2+M^2N^2)+C_{a,\gamma, \xi, \kappa}N M^2.
\end{align*}
Now, we drop the first term in the left hand side and we consider $t=0$. Let us write
\begin{equation*}
    A := \frac{\lambda_S (B_0(1))}{(1+ \kappa)(1 + \xi )^2} - \frac{a}{2}\Big(\frac\gamma 2+1\Big) - \frac{C_*M}{N}
\end{equation*}
and
\begin{equation*}
    B := C_{a,\gamma}(M^4+SMN^2+M^2N^2)+C_{a,\gamma, \xi, \kappa}N M^2.
\end{equation*}
Thus, we study
\begin{equation*}
     \frac{d}{d s}(g_\gamma(\cdot,0))(s) + \frac{A}{\theta_a(s)^{2}} g_\gamma (s,0) 
    \leq   \frac{B}{\theta_a(s)^{2}}.
\end{equation*}
We multiply this equation by the function 
\begin{equation*}
    k(s) = \left( \frac{\theta_a(s)}{\theta_a(s_0)} \right)^{-\frac{2A}{a}},  
\end{equation*}
where $s_0:=-\frac1{32(N+1)^2}$. The function $k(s)$ satisfies
\begin{equation*}
    \frac{dk}{ds}(s) =  \frac{A}{\theta_a(s)^{2}} \left( \frac{\theta_a(s)}{\theta_a(s_0)} \right)^{-\frac{2A}{a}} =  \frac{A}{\theta_a(s)^{2}} k(s).  
\end{equation*}
Then, we find
\begin{equation*}
    \frac{d}{ds}\left( kg_\gamma(\cdot, 0) \right) (s)
    \leq   \frac{B}{A} \left( \frac{A}{\theta_a(s)^{2}} k(s)   \right).
\end{equation*}
Let $s\in [s_0,0)$. Integrating over $[s_0,s]$ and multiplying by $k^{-1}(s)$ we obtain
\begin{equation}
\label{cfg}
    g_\gamma(s,0)
    \leq g_\gamma(s_0,0) \frac{1}{k(s)} + \frac{B}{A} \left( 1- \frac{1}{k(s)} \right).
\end{equation}
Estimate \eqref{cfg} will be useful in the next section to demonstrate Proposition \ref{prop.scaleinv}. 
We want $A>0$ in order to have a damping effect, which will be achieved by taking $\xi,\, \kappa$ small and $N$ large.

\subsection{Proof of Proposition \ref{prop.scaleinv}}
\label{subsec.scaleinv}

Our objective is to prove the boundedness of $f$. Here we take $\gamma=1$. We rely on the estimate \eqref{cfg}. As $a=\lambda_S(B_0 (1) )$ according to the choice made in Theorem \ref{theom.main},\footnote{Here we see that any $0 < a < \frac43\lambda_S(B_0 (1) )$ works; see Footnote \ref{foot.a}. We remark here that there is a possibility to get the full range $0 < a <4\lambda_S(B_0 (1) )$ as in \cite{Neus14}. Indeed, it suffices to choose different parameters in Young's inequality leading to \eqref{e.controlKII} and \eqref{e.controlKI}. This results in a small parameter $\varepsilon$ in front of the third term in the right hand side of \eqref{e.controlKII} and of the second term in the right hand side of \eqref{e.controlKI}. We do not carry out this technical modification here so as to keep the number of parameters to a minimum.} we first take $ 0 < \kappa, \xi \ll 1$ and then the parameter $N(M,\xi,\kappa)$ (see the definition \eqref{e.eqPsi} of $\Psi_N$) large enough to get 
\begin{equation}\label{e.choiceN}
    A = \frac{\lambda_S (B_0(1))}{(1+ \kappa)(1 + \xi )^2} - \frac{3a}{4} - \frac{C_*M}{N} >0.
\end{equation}
Hence, from \eqref{cfg}, we obtain
$1 / k( s)$ goes to $0$ when $ s \uparrow 0$. 
Thus, we get $g_1(\cdot,0)$ is bounded on $[s_0,0)$ with as defined above $s_0=-\frac1{32(N+1)^2}$.

Now, we want to prove that $f(s)$ is bounded. Take any $t_1\in[\frac{s_0}{2},0)$ and define $ s _1:= 2t_1  $. Thus, $s_1 < t_1 < 0$ and
\begin{equation*}
    2(t_1-s_1) = -s_1 = \frac{\theta_a(s_1)^2}{a}.
\end{equation*}
Moreover, for $\tau \in (s_1, t_1)$ we have
\begin{equation*}
    \frac{1}{\sqrt{2}} \theta_a(s_1) = \theta_a(t_1) < \theta_a(\tau) < \theta_a(s_1) = \sqrt{2}\theta_a (t_1).
\end{equation*}
Then, we get 
\begin{align*}
    \frac{1}{(t_1 -  s_1 )} \int\limits_{ s _1 }^{t_1}  f(\tau) d \tau &\leq \frac{  2a }{ \theta_a( s_1 )^2 } \int\limits_{ s _1 }^{t_1}  f(\tau) d \tau \leq 2a    \int\limits_{ s _1 }^{0}  \frac{1 
    }{ \theta_a(  s_1 ) ^{2}  } f(\tau) d \tau  \nonumber \\
    & =  2 a 
    g_1(  s_1,0 ) .
\end{align*}
We can find $s_1' \in ( s_1, t_1)$ such that
\begin{equation}\label{e.pigeonhole}
    f(s_1') \leq \frac{1}{(t_1 -  s_1 )} \int\limits_{ s_1 }^{t_1}  f(\tau) d \tau \leq 2  a      \sup_{s\in[s_0,0)}  
    g_1(s,0 ).
\end{equation}
Inequality \eqref{ee3m}, with $t=t_1$ and $s=s_1'$, gives
\begin{align*}
    \frac{1}{2
    } \,  f(t_1) \leq & \frac{ \theta_a (t_1)^ {2} }{ \theta_a (s_1 ') ^{2} } f(t_1) \nonumber\\
    \leq &     f(s_1 ') + \frac{C_* M}{N}g_1(s_1 ' ,0)+ C_{a,\gamma}(M^4+SMN^2+M^2N^2)\nonumber\\
    \leq &    2a\sup_{s\in[s_0,0)}g_1(s,0 ) + \frac{C_* M}{N}g_1(s_1 ' ,0)+ C_{a,\gamma}(M^4+SMN^2+M^2N^2),
\end{align*}
where we used \eqref{e.pigeonhole} in the last line. 
Therefore $f$ is bounded on a small non empty interval $[\frac{s_0}{2},0)$, and eventually on $(-1,0)$ by using the boundedness of the energy for $s\in(-1,\frac{s_0}{2})$, i.e.
\begin{equation}\label{e.bddf}
\esssup_{s\in(-1,0)}\frac{1}{ \theta_a(s) } \| \Psi_N v(\cdot,s) \|_2^2<\infty.
\end{equation}
This directly implies \eqref{controldef} which concludes the proof of Proposition \ref{prop.scaleinv}.

\section{Proof of the main results}
\label{mainproofs}

\subsection{Case with axisymmetry}

We prove \eqref{e.blowupL^3infty} in Theorem \ref{theom.main} by contraposition. Hence, we assume by contraposition that
\begin{equation}\label{e.prooftheommainL3infty}
\esssup_{s\in(-1,0)}\|v(\cdot,s)\|_{L^{3,\infty}(B_0(\sqrt{a})\setminus B_0(\sqrt{a(-s)}))}\leq M.
\end{equation}
There are two steps in this proof. Note that we only assume axisymmetry at the end of the second step.

\subsubsection*{Step 1: control of a scale-invariant Morrey-type quantity}
This step works without the assumption of axisymmetry. We start form the control \eqref{e.bddf}. Observe that for $\gamma > -1$,
\begin{equation}
\label{hypforene}
    \esssup_{s\in (-1,0)}  \int\limits_{s}^{0} \frac{ \theta_a(\tau)^{ \gamma - 2 } } { \theta_a(s) ^{\gamma +1} }   \| \Psi_N v (\tau) \|_{L^2 } ^{2} \, d \tau  < +\infty.
\end{equation}
By \eqref{hypforene} we can now repeat the computations that lead to 
the estimate \eqref{ee3} with $\gamma = 0$. Hence, we have \eqref{ee3} with $\gamma = 0$ and $t=0$, i.e.
\begin{align*}
    &\esssup_{s \in ( -\frac{1}{16(N+1)^2} ,0 ) } \frac{1} { \theta_a (s) }   \int\limits_{s}^{0} \| \nabla ( \Psi_N v ) \|_2^2\\
    &\leq \esssup_{s \in ( -\frac{1}{16(N+1)^2} ,0 )}\Bigg(\frac{1}{\theta_a(s) } \| \Psi_N v (s) \|_2^{2} +  \frac{C_*M}{N} \int\limits_{s}^{0}  \frac{ \theta_a(\tau)^{- 2 }  }{ \theta_a(s)} \| \Psi_N v \|_2^{2}\, d\tau\Bigg)
     \nonumber \\
      &\qquad +C_{a,\gamma}(M^4+SMN^2+M^2N^2).
      \end{align*}
Therefore,
\begin{equation*}
    \esssup_{s \in (  -\frac{1}{16(N+1)^2} ,0 ) }  \frac{1}{\theta_a(s)}\| { \Psi_N v} (s) \|^2_{L^2} +  \frac{ 1} { \theta_a (s) }  \int\limits_{s}^{0}   \| \nabla ( \Psi_N v ) \|_2^2 \, d \tau  < + \infty.
\end{equation*}
By interpolation between $L^2$ and $L^6$, this implies that
\begin{equation}
\label{key1}
    \esssup_{s \in (  -\frac{1}{16(N+1)^2} ,0 ) } \frac{ 1}{\theta_a(s)^2} \int\limits_{s}^0 \int\limits_{B_0 ( {\sqrt{a}} ) } | \Psi_N { v }(x,\tau) |^3    \, dx \, d\tau < + \infty.
\end{equation}
Therefore,
\begin{equation}
\label{key2}
    \esssup_{s \in (  -\frac{1}{16(N+1)^2} ,0 ) } \frac{ 1}{\theta_a(s)^2} \int\limits_{s}^0 \int\limits_{B_0 (\theta_a(\tau))} | v(x,\tau) |^3    \, dx \, d\tau < + \infty.
\end{equation}
Moreover, the assumption \eqref{e.prooftheommainL3infty} implies that for $\delta\in(0,\frac15)$
\begin{equation}
\label{key3}
    \esssup_{s \in (  -\frac{1}{16(N+1)^2} ,0 )} \frac{ 1}{\theta_a(s)^{5(1-\frac{1+\delta}{2})}}\int\limits_{s}^0 \int\limits_{B_0 (\theta_a(s))\setminus B_0 (\theta_a(\tau))} | v(x,\tau) |^{\frac52(1+\delta)}    \, dx \, d\tau < + \infty.
\end{equation}
Combining \eqref{key2} and \eqref{key3} yields
\begin{equation}
\label{key4}
    \esssup_{r \in (0,\frac{1}{4(N+1)})} \frac{ 1}{r^{5(1-\frac{1+\delta}{2})}} \int\limits_{-r^2}^0 \int\limits_{B_0 (r)} | v |^{\frac52(1+\delta)}    \, dx \, d\tau < + \infty.
\end{equation}

\subsubsection*{Step 2: conclusion}
By \cite[Theorem 6]{Seregin2018} (see also \cite{SerZajac}), we have that \eqref{key4} for $0<\delta\ll 1$ implies the boundedness of 
\begin{multline}\label{e.typeI}
\esssup_{r \in (0,\frac{1}{4(N+1)})}\Bigg(\esssup_{s \in (  -r^2,0 )} \frac{ 1}{r}\int\limits_{B_0 (r)} | v(x,\tau) |^2    \, dx \, d\tau+\frac{ 1}{r^2} \int\limits_{-r^2}^0 \int\limits_{B_0 (r)} | v |^{3}    \, dx \, d\tau\\
+\frac{ 1}{r} \int\limits_{-r^2}^0 \int\limits_{B_0 (r)} |\nabla v|^{2}\, dx \, d\tau\Bigg)  < + \infty.
\end{multline}
Let us now assume that $v$ is axisymmetric. Then \eqref{e.typeI} and \cite[Theorem 2.1]{seregin2020local} imply that $(0,0)$ is a regular point.

\subsection{General case without axisymmetry}
\label{sec.L3final}

We prove \eqref{e.blowupL^3} in Theorem \ref{theom.main} by contraposition. Hence, we assume by contraposition that
\begin{equation}\label{e.prooftheommainL3bis}
\esssup_{s\in(-1,0)}\|v(\cdot,s)\|_{L^{3}(B_0(\sqrt{a})\setminus B_0(\sqrt{a(-s)}))}\leq M.
\end{equation}
There are two steps in this proof. 

\subsubsection*{Step 1: control of a scale-invariant Morrey-type quantity}
Note that \eqref{e.prooftheommainL3bis} implies \eqref{e.prooftheommainL3infty}. Hence estimate \eqref{key1} holds. 
Combining \eqref{key1} with \eqref{e.prooftheommainL3bis} gives
\begin{equation}
\label{hyppr}
    \esssup_{s \in (  -\frac{1}{16(N+1)^2} ,0 ) }  \frac{ 1}{\theta_a(s)^2} \int\limits_{s}^0 \int\limits_{B_0 ( \sqrt{a})  } | { v }  (x,\tau )|^3   \, dx \, d\tau < +\infty .
\end{equation}

\subsubsection*{Step 2: conclusion}
From \eqref{hyppr} and a pigeonholing argument, we find a sequence of times $t_k\in(-\frac{1}{16(N+1)^2} ,0 )$ such that $t_k\uparrow 0$ and 
$$
\sup_k\int\limits_{B_0 ( \sqrt{a})  } |v(x,t_k )|^3  \, dx<\infty.
$$
Applying \cite[Theorem 1.1]{AlbrittonBarker2018local} immediately yields that $(0,0)$ is a regular point.

\appendix

\section{Known results about Lorentz spaces}
\label{sec.lorentz}

Given a measurable subset $\Omega\subseteq\mathbb{R}^{d}$, let us define the Lorentz spaces. 
For a measurable function $f:\Omega\rightarrow\mathbb{R}$ define:
\begin{equation}\label{defdistchapter2}
d_{f,\Omega}(\alpha):=\mu(\{x\in \Omega : |f(x)|>\alpha\}),
\end{equation}
where $\mu$ denotes the Lebesgue measure.
 The Lorentz space $L^{p,q}(\Omega)$, with $p\in [1,\infty)$, $q\in [1,\infty]$, is the set of all measurable functions $g$ on $\Omega$ such that the quasinorm $\|g\|_{L^{p,q}(\Omega)}$ is finite. Here:

\begin{equation}\label{Lorentznormchapter2}
\|g\|_{L^{p,q}(\Omega)}:= \Big(p\int\limits_{0}^{\infty}\alpha^{q}d_{g,\Omega}(\alpha)^{\frac{q}{p}}\frac{d\alpha}{\alpha}\Big)^{\frac{1}{q}},
\end{equation}
\begin{equation}\label{Lorentznorminftychapter2}
\|g\|_{L^{p,\infty}(\Omega)}:= \sup_{\alpha>0}\alpha d_{g,\Omega}(\alpha)^{\frac{1}{p}}.
\end{equation}\\
It is known there exists a norm, which is equivalent to the quasinorm defined above, for which $L^{p,q}(\Omega)$ is a Banach space. 
For $p\in [1,\infty)$ and $1\leq q_{1}< q_{2}\leq \infty$, we have the following continuous embeddings 
\begin{equation}\label{Lorentzcontinuousembeddingchapter2}
L^{p,q_1}(\Omega) \hookrightarrow  L^{p,q_2}(\Omega)
\end{equation}
and the inclusion is known to be strict.

Our main tool in this paper is the following H\"older's inequality for Lorentz spaces. The statement below and proof can be found in Hunt's paper \cite[Theorem 4.5, p.271]{hunt}; see
also \cite[Theorems 3.4-3.5, page 141]{O'Neil}.
\begin{Proposition}\label{hunt}
Suppose that $1\leq p,q,r\leq\infty$ and $1\leq s_1,s_2\leq\infty$. Furthermore, suppose that $p$, $q$, $r$, $s_1$ and $s_2$ satisfy the following relations: $$\frac{1}{p}+\frac{1}{q}=\frac{1}{r}$$ and $$\frac{1}{s_1}+\frac{1}{s_2}=\frac{1}{s}.$$ 
Then the assumption that $f\in L^{p,s_1}(\Omega)$ and $g\in L^{q,s_2}(\Omega)$ imply that $fg \in L^{r,s}(\Omega)$, with the estimate 
\begin{equation}\label{Holderverygeneral}
\|fg\|_{L^{r,s}(\Omega)}\leq C(p,q,s_1,s_2)\|f\|_{L^{p,s_1}(\Omega)}\|g\|_{L^{q,s_2}(\Omega)}.
\end{equation}
\end{Proposition}

\section*{Conflict of interest}
The authors declare that they have no conflict of interest.

\section*{Data availability statement}

Data sharing is not applicable to this article as no datasets were generated or analyzed during the current study.

\section*{Acknowledgment}
CP and PFD are partially supported by the Agence Nationale de la Recherche,
project BORDS, grant ANR-16-CE40-0027-01. CP is also partially supported by the Agence Nationale de la Recherche,
 project SINGFLOWS, grant ANR-
18-CE40-0027-01, project CRISIS, grant ANR-20-CE40-0020-01, by the CY
Initiative of Excellence, project CYNA (CY Nonlinear Analysis) and project CYFI (CYngular Fluids and Interfaces). PFD is also supported by the Labex MME-DII. TB and CP thank the Institute of Advanced Studies of Cergy Paris University for their hospitality. 

\small 
\bibliographystyle{abbrv}
\bibliography{pedro.bib}

\begin{thebibliography}{10}

\bibitem{AlbrittonBarker2018local}
D.~Albritton and T.~Barker.
\newblock Localised necessary conditions for singularity formation in the
  {N}avier-{S}tokes equations with curved boundary.
\newblock {\em J. Differential Equations}, 269(9):7529--7573, 2020.

\bibitem{TB22}
T.~{Barker}.
\newblock {Localized quantitative estimates and potential blow-up rates for the
  {N}avier-{S}tokes equations}.
\newblock {\em arXiv e-prints}, page arXiv:2209.15627, Sept. 2022.

\bibitem{BP20norm}
T.~Barker and C.~Prange.
\newblock Localized {S}moothing for the {N}avier-{S}tokes {E}quations and
  {C}oncentration of {C}ritical {N}orms {N}ear {S}ingularities.
\newblock {\em Arch. Ration. Mech. Anal.}, 236(3):1487--1541, 2020.

\bibitem{BP21cmp}
T.~Barker and C.~Prange.
\newblock Quantitative regularity for the {N}avier-{S}tokes equations via
  spatial concentration.
\newblock {\em Comm. Math. Phys.}, 385(2):717--792, 2021.

\bibitem{BP22survey}
T.~{Barker} and C.~{Prange}.
\newblock {From concentration to quantitative regularity: a short survey of
  recent developments for the Navier-Stokes equations}.
\newblock {\em arXiv e-prints}, page arXiv:2211.16215, Nov. 2022.

\bibitem{BL76}
J.~Bergh and J.~L\"{o}fstr\"{o}m.
\newblock {\em Interpolation spaces. {A}n introduction}.
\newblock Grundlehren der Mathematischen Wissenschaften, No. 223.
  Springer-Verlag, Berlin-New York, 1976.

\bibitem{BG17}
Z.~Bradshaw and Z.~Gruji\'{c}.
\newblock Frequency localized regularity criteria for the 3{D}
  {N}avier-{S}tokes equations.
\newblock {\em Arch. Ration. Mech. Anal.}, 224(1):125--133, 2017.

\bibitem{bradshaw2022local}
Z.~Bradshaw and T.-P. Tsai.
\newblock On the local pressure expansion for the {N}avier-{S}tokes equations.
\newblock {\em Journal of Mathematical Fluid Mechanics}, 24(1):1--32, 2022.

\bibitem{CKN82}
L.~Caffarelli, R.~Kohn, and L.~Nirenberg.
\newblock Partial regularity of suitable weak solutions of the
  {N}avier-{S}tokes equations.
\newblock {\em Comm. Pure Appl. Math.}, 35(6):771--831, 1982.

\bibitem{chen2009lower}
C.-C. Chen, R.~M. Strain, T.-P. Tsai, and H.-T. Yau.
\newblock Lower bounds on the blow-up rate of the axisymmetric
  {N}avier-{S}tokes equations {II}.
\newblock {\em Communications in Partial Differential Equations},
  34(3):203--232, 2009.

\bibitem{chen2008lower}
C.-C. Chen, R.~M. Strain, H.-T. Yau, and T.-P. Tsai.
\newblock Lower bound on the blow-up rate of the axisymmetric {N}avier-{S}tokes
  equations.
\newblock {\em International Mathematics Research Notices}, 2008, 2008.

\bibitem{ChesDai19}
A.~Cheskidov and M.~Dai.
\newblock Kolmogorov's dissipation number and the number of degrees of freedom
  for the 3d {Navier}-{Stokes} equations.
\newblock {\em Proc. R. Soc. Edinb., Sect. A, Math.}, 149(2):429--446, 2019.

\bibitem{CS14}
A.~Cheskidov and R.~Shvydkoy.
\newblock A unified approach to regularity problems for the 3{D}
  {N}avier-{S}tokes and {E}uler equations: the use of {K}olmogorov's
  dissipation range.
\newblock {\em J. Math. Fluid Mech.}, 16(2):263--273, 2014.

\bibitem{ESS2003}
L.~Escauriaza, G.~A. Seregin, and V.~\v{S}ver\'{a}k.
\newblock {$L_{3,\infty}$}-solutions of {N}avier-{S}tokes equations and
  backward uniqueness.
\newblock {\em Uspekhi Mat. Nauk}, 58(2(350)):3--44, 2003.

\bibitem{Galdibook}
G.~P. Galdi.
\newblock {\em An introduction to the mathematical theory of the
  {N}avier-{S}tokes equations}.
\newblock Springer Monographs in Mathematics. Springer, New York, second
  edition, 2011.
\newblock Steady-state problems.

\bibitem{GrujicXu2019-dynrestr}
Z.~Gruji\'{c} and L.~Xu.
\newblock A regularity criterion for {3D} {NSE} in dynamicaly restricted local
  {M}orrey spaces.
\newblock {\em Applicable Analysis}, 0(0):1--15, 2021.

\bibitem{hunt}
R.~A. Hunt.
\newblock On {$L(p,q)$} spaces.
\newblock {\em Enseignement Math. (2)}, 12:249--276, 1966.

\bibitem{KangMiuraTsai20-concL2}
K.~Kang, H.~Miura, and T.-P. Tsai.
\newblock Regular sets and an {$\epsilon$}-regularity theorem in terms of
  initial data for the {N}avier-{S}tokes equations.
\newblock {\em Pure Appl. Anal.}, 3(3):567--594, 2021.

\bibitem{koch2009liouville}
G.~Koch, N.~Nadirashvili, G.~Seregin, and V.~{\v{S}}ver{\'a}k.
\newblock Liouville theorems for the {N}avier-{S}tokes equations and
  applications.
\newblock {\em Acta Mathematica}, 203(1):83--105, 2009.

\bibitem{lei2011liouville}
Z.~Lei and Q.~S. Zhang.
\newblock A {L}iouville theorem for the axially-symmetric {N}avier--{S}tokes
  equations.
\newblock {\em Journal of Functional Analysis}, 261(8):2323--2345, 2011.

\bibitem{LOW18}
K.~Li, T.~Ozawa, and B.~Wang.
\newblock Dynamical behavior for the solutions of the {N}avier-{S}tokes
  equation.
\newblock {\em Commun. Pure Appl. Anal.}, 17(4):1511--1560, 2018.

\bibitem{Luo19}
X.~Luo.
\newblock A {B}eale-{K}ato-{M}ajda criterion with optimal frequency and
  temporal localization.
\newblock {\em J. Math. Fluid Mech.}, 21(1):Art. 1, 16, 2019.

\bibitem{Neus12}
J.~Neustupa.
\newblock A removable singularity in a suitable weak solution to the
  {N}avier-{S}tokes equations.
\newblock {\em Nonlinearity}, 25(6):1695--1708, 2012.

\bibitem{Neus14}
J.~Neustupa.
\newblock A refinement of the local {S}errin-type regularity criterion for a
  suitable weak solution to the {N}avier-{S}tokes equations.
\newblock {\em Arch. Ration. Mech. Anal.}, 214(2):525--544, 2014.

\bibitem{NN02}
J.~Ne\v{c}as and J.~Neustupa.
\newblock New conditions for local regularity of a suitable weak solution to
  the {N}avier-{S}tokes equation.
\newblock {\em J. Math. Fluid Mech.}, 4(3):237--256, 2002.

\bibitem{O'Neil}
R.~O'Neil.
\newblock Convolution operators and ${L}(p,q)$ spaces.
\newblock {\em Duke Math. J.}, 30(1):129--142, 03 1963.

\bibitem{OP22}
W.~S. {O{\.z}a{\'n}ski} and S.~{Palasek}.
\newblock {Quantitative control of solutions to axisymmetric {N}avier-{S}tokes
  equations in terms of the weak ${L}^3$ norm}.
\newblock {\em arXiv e-prints}, page arXiv:2210.10030, Oct. 2022.

\bibitem{seregin2020local}
G.~Seregin.
\newblock Local regularity of axisymmetric solutions to the {N}avier--{S}tokes
  equations.
\newblock {\em Analysis and Mathematical Physics}, 10(4):46, 2020.

\bibitem{sersve2009type}
G.~Seregin and V.~{\v{S}}ver{\'a}k.
\newblock On {T}ype {I} singularities of the local axi-symmetric solutions of
  the {N}avier-{S}tokes equations.
\newblock {\em Communications in Partial Differential Equations},
  34(2):171--201, 2009.

\bibitem{Seregin2018}
G.~Seregin and V.~{\v{S}}ver{\'a}k.
\newblock Regularity criteria for {N}avier-{S}tokes solutions.
\newblock In Y.~Giga and A.~Novotn{\'y}, editors, {\em Handbook of Mathematical
  Analysis in Mechanics of Viscous Fluids}, pages 829--867. Springer
  International Publishing, Cham, 2018.

\bibitem{SerZajac}
G.~A. Seregin and W.~Zajaczkowski.
\newblock A sufficient condition of local regularity for the {N}avier-{S}tokes
  equations.
\newblock {\em Zap. Nauchn. Sem. S.-Peterburg. Otdel. Mat. Inst. Steklov.
  (POMI)}, 336(Kraev. Zadachi Mat. Fiz. i Smezh. Vopr. Teor. Funkts.
  37):46--54, 274, 2006.

\bibitem{Tao19}
T.~Tao.
\newblock Quantitative bounds for critically bounded solutions to the
  {N}avier-{S}tokes equations.
\newblock In A.~Kechris, N.~Makarov, D.~Ramakrishnan, and X.~Zhu, editors, {\em
  Nine Mathematical Challenges: An Elucidation}, volume 104. American
  Mathematical Society, 2021.

\end{thebibliography}

\end{document}